\documentclass[leqno,12pt]{article}
\usepackage[margin=2cm]{geometry}

\usepackage{graphicx}
\usepackage{amsmath,amsfonts,amssymb,amsthm}
\usepackage{enumitem}

\newtheorem{theorem}{Theorem}
\newtheorem{corollary}{Corollary}
\newtheorem{proposition}{Proposition}
\theoremstyle{definition}
\newtheorem{definition}{Definition}

\newtheorem{remark}{Remark}

\newcommand{\ir}[1]{\mathcal{#1}}
\newcommand{\RR}{\mathbb{R}}
\newcounter{unit}[section]
\def\theunit{\thesection.\arabic{unit}}
\newcommand{\unit}{\smallskip
\refstepcounter{unit}
\noindent\mbox{{\bf\theunit.\ }}}
\setitemize[0]{leftmargin=35pt,itemindent=00pt}

\newcommand{\grad}{\nabla}

\begin{document}

\title{Differentiable distance spaces\thanks{The second author was supported by the Hungarian Scientific Research Fund (OTKA) Grant K-111651}}

\author{L. Tam\'assy \and D. Cs. Kert\'esz 
 }

\maketitle

\begin{abstract}
The distance function $\varrho(p,q)$ (or $d(p,q)$) of a distance space (general metric space) is not differentiable in general. We investigate such distance spaces over $\RR^n$, whose distance functions are differentiable like in case of Finsler spaces. These spaces have several good properties, yet they are no Finsler spaces (which are special distance spaces). They are situated between general metric spaces (distance spaces) and Finsler spaces. We will investigate such curves of differentiable distance spaces, which possess the same properties as geodesics do in Finsler spaces. So these curves can be considered as forerunners of Finsler geodesics. They are in greater plenitude than Finsler geodesics, but they become geodesics in a Finsler space. We show some properties of these curves, as well as some relations between differentiable distance spaces and Finsler spaces. We arrive to these curves and to our results by using distance spheres, and using no variational calculus. We often apply direct geometric considerations.
{\bf MSC class:} 53B40,   51K05, 51K99, 54E35.
\end{abstract}

\section{Introduction}
Distance spaces (general metric spaces) were widely investigated from different point of views, but not with differentiable distance functions $\varrho$. In this paper we investigate distance spaces with differentiable distance functions $\varrho$. These differentiable distance spaces $D^n$ lie between general metric spaces (distance spaces) and Finsler spaces $F^n$, which are special distance spaces. However, these $D^n$ are still far from Finsler spaces. We construct geodesic curves of an $F^n$ as osculation points of geodesic spheres. This construction can be performed also in our $D^n$. The resulting curves are called osculation curves. If our $D^n$ reduces to an $F^n$, then osculation curves become Finsler geodesics. So osculation curves of our $D^n$ can be considered as forerunners of Finsler geodesics. Actually a number of osculation curves become a single geodesic. Under certain conditions osculation curves become quasigeodesics, which are already near to, but not completely same as Finsler geodesics. We show some properties of these curves. Under certain differentiability conditions, every $D^n$ determines an $F^n$, but this relation is not 1:1. Many $D^n$ determine the same $F^n$. We show some relations between differentiable distance spaces and Finsler spaces. Finally, we obtain a theorem on projectively flat Finsler spaces. We often apply distance spheres, convexity and direct geometric considerations, but we do not use variational calculus.

\section{Preliminaries}\label{sec:pre}
 We recall some simple facts on Finsler geometry. The basic notion of a Finsler space $F^n=(M,\ir F)$ over the base manifold $M$ and with the Finsler metric $\ir F$ is the arc length $s_F(x(t))$ of a curve $x(t)$, $t\in I=[\alpha,\beta]$:
\begin{equation}\label{eq:arcint}
  s_F(x(t)):=\int_\alpha^\beta\ir F(x(t),\dot x(t))dt.
\end{equation}
$\ir F$ is a function on $TM$:
\[
 \ir F\colon TM\to\RR_+,\quad (p,y)\longmapsto \ir F(p,y),\quad p\in M,\ y\in T_pM,
\]
with the defining properties
\begin{itemize}
  \item[(F1)] \emph{$\ir F$ is $C^0$ on $TM$ and $C^\infty$ on $TM\setminus\{0\}$;}
  \item[(F2a)] \emph{$\ir F(p,\lambda y)=|\lambda| F(p,y)$ if $\lambda\in\RR$ (absolute homogeneity);} or
  \item[(F2b)] \emph{$\ir F(p,\lambda y)=\lambda F(p,y)$ if $\lambda\in\RR_+$ (positive homogeneity);}
  \item[(F3)] \emph{$\frac{\partial^2 \ir F^2}{\partial y^i\partial y^j}(p,y)$ are the coefficients of a positive definite quadratic form.}
\end{itemize}
The indicatrix
\[
 I(p_0):=\{y\mid \ir F(p_0,y)=1\}
\]
of $F^n$ at $p_0$ is a hypersurface (level surface) of $z=\ir F(p_0,y)$ in $T_{p_0}M$. It is strictly convex in the sense that it bounds a strictly convex open set. (F2a) means that the indicatrices are symmetric, while from (F2b) this does not follow. The Finsler norm of $y\in T_{p_0}M$ is
\[
 |y|_F:=\ir F(p_0,y).
\]
In case of (F2a) by this norm $T_{p_0}M$ becomes a Banach space. Dividing $x(t)$ into $N$ small parts $dx$, the arc length $s_F$ intuitively can be obtained in the following way
\begin{equation}\label{eq:arcintdx}
 s_F(x(t))\approx\sum|dx|_F=\sum\ir F(x,dx)=\sum\ir F(x,\frac{dx}{dt})dt\longrightarrow \int_\alpha^\beta\ir F(x(t),\dot x(t))dt.\tag{\ref{eq:arcint}'}
\end{equation}
The decisive important (F2a) is equivalent to the property that the arc length $s_F(x(t))$ is independent of the reparametrization of the curve $x(t)$ including the change of the orientation. In this case the metric is called reversible, while in the more general case of (F2b) it is irreversible. The Finsler distance of two points $a,b\in M$ is given by
\begin{equation}\label{eq:Fdist}
  \varrho^F(a,b)=\inf_{x\in\Gamma}\int^\beta_\alpha\ir F(x(t),\dot x(t))dt
\end{equation}
where $\Gamma$ means the collection of the curves $x(t)$ connecting $a$ and $b$. If in \eqref{eq:Fdist} $\inf$ is attained by a curve, then this curve is an extremal, a minimizing geodesic.

In contrary to Finsler spaces, in a distance space $D^n=(M,\varrho)$ the basic notion is the distance $\varrho(a,b)$ given by the distance function
\[
 \varrho\colon M\times M\to\RR_+,\quad (a,b)\longmapsto \varrho(a,b).
\]
[In a way similar to \eqref{eq:arcintdx} one can define also ``distance arc length'' $s_D(x(t))$]. $\varrho$ has to obey to the rules
\begin{itemize}
\item[(D1)] \emph{$\varrho(a,b)\ge 0$, and $\varrho(a,b)=0\quad \Longleftrightarrow\quad a=b$.}
\item[(D2)] \emph{$\varrho(a,b)=\varrho(b,a)$.}
\item[(D3)] \emph{$\varrho(a,c)+\varrho(c,b)\geq\varrho(a,b)$.}
\end{itemize}

Both Finsler and distance spaces were widely investigated. We refer here on the books \cite{BCS00:intr}, \cite{blumenthal1953theory}, \cite{MR1835418} and \cite{MR3156529} only. These books contain a vast old and up-to-date bibliography of works on different aspects of the geometry of distance and Finsler spaces. Also Busemann's G-spaces \cite{Bu2} represents an important class of distance spaces; in his dissertation M. G. Knecht \cite{knecht2006construction} found a quite weak curvature condition, which guarantees that a length metric space carries a differential structure and a Finsler metric. The Finsler distance of a reversible Finsler space satisfies (D1)--(D3). So every reversible Finsler space is a distance space, but not conversely. One of the main difference between Finsler and distance spaces is the lack of the differentiability at distance spaces. We suppose the differentiability of the distance function $\varrho$. Such distance spaces are nearer to Finsler spaces, yet they still considerably differ from them.
      
Our investigations will be local. So we can assume that the base manifold is $\RR^n$. In the most cases we will suppose that the distance spheres and the geodesic spheres are strictly convex in $\RR^n$. In the most cases also the Finsler spaces are supposed to be reversible.  

In section~\ref{sec:Finsler} we show that in a Finsler space $F^n=(\RR^n,\ir F)$ the strict convexity of the geodesic spheres implies that geodesics between two points are unique (Theorem~\ref{thm:convex}). Also we construct Finsler geodesics by geodesic spheres in two different ways. In section~\ref{sec:osc} we consider osculation points of two distance spheres with fix centers and varying radii. These points form an osculation curve $o(r)$. If $\varrho=\varrho^F$ (i.e. $D^n=F^n$), then these curves are Finsler geodesics. Osculation curves have some properties common with Finsler geodesics, but they have not only such ones. In contrary to Finsler geodesics, a segment of an osculation curve need not to be an osculation curve. An osculation curve whose segments are again osculation curves will be called quasigeodesic. We find their properties, which are very similar to that of Finsler geodesics (Theorems~\ref{thm:qC} and \ref{thm:q}). In the last section~\ref{sec:Fins} we investigate the relation of distance spaces to Finsler spaces. We show that a distance space determines a `weak' Finsler space (whose indicatrices may be only convex and not necessarily strictly convex). (Theorem~\ref{thm:2}). But this relation is not 1:1, since different distance spaces may determine the same Finsler space. The distance arc length and the Finsler arc length of a curve equal (Theorem~\ref{thm:3}), but this does not yield $D^n=F^n$. For this a necessary and sufficient condition is given in Theorem~\ref{thm:4}. Finally we show that a projectively flat $F^n=(\RR^n,\ir F)$ with geodesic spheres, which are symmetric in $\RR^n$, is a Minkowski space (Theorem~\ref{thm:5}).

\section{Finsler geodesics and geodesic spheres}\label{sec:Finsler}
We show some relations between Finsler geodesics and geodesic spheres including the definition of the Finsler geodesic by geodesic spheres. First we recall a classical definition of shortest geodesics in a metric space, especially in a Finsler space.

\unit Let $g(t)$, $t\in I=[\alpha,\beta]$ be a shortest geodesic of a Finsler space, and $g[a,b]$ an arc of it. Then
\begin{subequations}
  \begin{equation}\label{eq:2a}
\varrho^F(a,x)+\varrho^F(x,b)=\varrho^F(a,b),\qquad \forall\,x\in g[a,b].
  \end{equation}
This is the additivity property of the geodesics. Nevertheless, \eqref{eq:2a} alone does not assure that $g(t)$ is a geodesic arc. To this it is necessary that \eqref{eq:2a} be satisfied for any segment of $g(t)$, that is
\begin{equation}\label{eq:2b}
\varrho^F(g(t_1), g(t))+\varrho^F(g(t),g(t_2))=\varrho^F(g(t_1),g(t_2)),\qquad \forall\,t_1\leq t\leq t_2\mbox{ in } I,
\end{equation}
\end{subequations}
which contains \eqref{eq:2a} as a special case. \eqref{eq:2b} is a characteristic property of the shortest Finsler geodesics. 

\unit{} A geodesic sphere of a Finsler space $F^n=(\RR^n,\ir F)$ centered at the point $a\in\RR^n$ with radius $r$ is defined by
\[
 S_a(r):=\{q\in\RR^n\mid \varrho^F(a,q)=r\}.
\]
$S_a(r)$ is a smooth hypersurface of $\RR^n$, and diffeomorphic to the Euclidean unit sphere $S^{n-1}\subset E^n$. We show that the strict convexity of the geodesic spheres implies that between two points exists at most one geodesic.

\begin{theorem}\label{thm:convex}
 (a) If the geodesic spheres of a Finsler space $F^n=(\RR^n,\ir F)$ are strictly convex, then between any pair of points $a$, $b$, there exists at most one geodesic.
 
 (b) If between $a$ and $b$ there exist two geodesics $g_1(s)$ and $g_2(s)$, then among the geodesic spheres centered at $a$ or $b$ there exist infinitely many not strictly convex ones.
\end{theorem}

\begin{proof}
 (a) We suppose that the geodesic spheres are strictly convex in $\RR^n$, and that between $a$ and $b$ there exists a geodesic arc $g[a,b]$. Let $r$ be smaller than the Finsler distance between $a$ and $b$: $r<\varrho^F(a,b)=\Delta$. Denote by $\bar r$ the smallest value, for which $S_a(r)$ (with a fix $r$) and $S_b(\bar r)$ still have a common point $p_0$. Because of the strict convexity of $S_a(r)$ and $S_b(\bar r)$, $p_0$ is unique. These two geodesic spheres can not be intersecting, for then $\bar r$ would not be minimal. So these geodesic spheres are osculating from outside at $p_0$. 
 
 We know that for the points $q\in\RR^n$,
 \begin{equation}\label{eq:Fdisttriang}
 \varrho^F(a,q)+\varrho^F(q,b)\geq\varrho^F(a,b),
 \end{equation}
 and we obtain equality  in \eqref{eq:Fdisttriang} exactly for the points $p\in g[a,b]$. Then
 \begin{equation}\label{eq:pqmin}
  \varrho^F(a,p)+\varrho^F(p,b)=\min_{q\in\RR^n}\big\{\varrho^F(a,q)+\varrho^F(q,b)\big\}.
 \end{equation}
 If we confine $q$ to $S_a(r)$, then
 \begin{equation}
  \min_{q\in S_a(r)}\big\{\varrho^F(a,q)+\varrho^F(q,b)\big\}
 \end{equation}
 occurs at the unique osculation point $p_0$ of $S_a(r)$ and $S_b(\bar r)$. Hence $q[a,b]\cap S_a(r)$ is the unique $p_0$. This is true for any $r\in[0,\Delta]$. This shows the unicity of $g[a,b]$.
 
  (b) Suppose that between $a$ and $b$ there exist two different geodesics $g_1[a,b]$ and $g_2[a,b]$. Then there exist $r\in(0,\Delta)$  such that
  \[
   S_a(r)\cap g_1[a,b]=p_1\neq p_2=S_a(r)\cap g_2[a,b].
  \]
  $p_1$ is a point of $S_a(r)$, $r=\varrho^F(a,p_1)$, and also of $S_b(\bar r)$, $\bar r=\Delta-\varrho^F(a,p_1)$. These two geodesic spheres have no common inner point $p$, namely in this case we would have $\varrho^F(a,p)+\varrho^F(p,b)<\Delta$, which contradicts \eqref{eq:Fdisttriang}. Therefore $S_a(r)$ and $S_b(\bar r)$ are geodesic spheres osculating at $p_1$. The same holds at the point $p_2$. $S_a(r)$ can be strictly convex only if the straight line segment $\ell(p_1,p_2)$ between $p_1$ and $p_2$ lies in the inside of $S_a(r)$ (except $p_1$ and $p_2$). In this case $\ell(p_1,p_2)$ can not be in $S_b(\bar r)$ since $S_a(r)$ and $S_b(\bar r)$ have no common inner point (they are osculating from outside). In this case $S_b(\bar r)$ can not be convex. We obtain a similar result if we suppose $\ell(p_1,p_2)$ to be in $S_b(\bar r)$. So if there exist two geodesics $g_1[a,b]$ and $g_2[a,b]$, then either $S_a(r)$ or $S_b(\bar r)$ is not strictly convex. This is true for every $r\in(0,\Delta)$. Thus there exist infinitely many not strictly convex geodesics spheres centered at $a$ and $b$.
\end{proof}

In part (a) we have shown that convexity yields the unicity. We give a simple example, which shows that the lack of convexity may cause the lack of unicity. Let $a^*$ and $b^*$ be antipodal points of $S^n\subset E^{n+1}$. Then the (geodesic) spheres $S_{a^*}(r^*)$ and $S_{b^*}(\bar r^*)$ with radii $r^*+\bar r^*=\pi$ coincide. Then also their stereographic images $S_a(r)$ and $S_b(\bar r)$ are coinciding geodesic spheres in the inherited metric. However the centers $a$ and $b$ are on different sides of the coinciding geodesic spheres $S_a(r)$ and $S_b(\bar r)$. So if $S_a(r)$ is convex in $\RR^n$, then $S_b(\bar r)$ is not so, or conversely, and in the same time the images of the numerous great circles (i.e., geodesics of $S^n$) between $a^*$ and $b^*$ become geodesics between $a$ and $b$ in the inherited metric.

\unit \label{ssec:Fosc}  We show that points of a geodesic curve of a Finsler space $F^n=(\RR^n,\ir F)$ with strictly convex geodesic spheres can be characterized (defined) as osculation points of geodesic spheres centered at its two arbitrarily chosen points.

Let $g(s)$ with arc length parameter $s\in(-\infty,\infty)$ be a geodesic curve of $F^n$ with $g(0)=a$, $g(\Delta)=b$, and  $g(r)$ with $0<r<\Delta$ a third point of $g(s)$ between $a$ and $b$. Denoting $\bar r=\varrho^F(g(r),b)=\Delta-r$ we obtain that $g(r)$ is a common point of $S_a(r)$ and $S_b(\bar r)$. These geodesic spheres again must be osculating and cannot be intersecting (see the proof of Theorem~\ref{thm:convex}). Then $g(r)$ is an osculation point of $S_a(r)$ and $S_b(\bar r)$. Because of the strict convexity of the geodesic spheres, to $S_a(r)$ there exists a unique $S_b(\bar r)$ osculating from outside, and for the same reason also the osculation point $c(r)$ is unique. So the osculation point $c(r)$ is exactly $g(r)$. This is true for any $r\in [0,\Delta]$.

Now let $r>\Delta$. Then $a=g(0)$, $g=b(\Delta)$ and $g(r)=c$ are three points on $g(s)$ in this arrangement. $c$ is again a common point of $S_a(r)$ and $S_b(\bar r)$, $\bar r= r-\Delta$. $S_a(r)$ and $S_b(\bar r)$ must osculate at $c$, that is, they can not have more common points. Namely suppose that they have still another common point $p$. Then the triangle $(a,b,p)$ yields the inequality
\begin{equation}\label{eq:abp}
 \varrho(a,b)+\varrho(b,p)>\varrho(a,p)=r,
\end{equation}
while on the geodesic $g(s)$ we have
\begin{equation}\label{eq:abc}
 \varrho(a,b)+\varrho(b,c)=\varrho(a,c)=r.
\end{equation}
But $\varrho(b,p)=\varrho(b,c)$, for $p,c\in S_b(\bar r)$. Thus \eqref{eq:abp} and \eqref{eq:abc} are contradicting, and hence $c$ is the unique common point of $S_a(r)$ and $S_b(\bar r)$.

This is true for any $r\in (\Delta,\infty)$. So any $g(r)$, $r>0$ is an osculation point of two geodesic spheres centered at $a$ and $b$ respectively, and any osculation point $p$ is a point of $g(s)$. Interchanging the role of $a$ and $b$ we obtain that the above statements are true for the whole $g(s)$, $s\in(-\infty,\infty)$. Thus this property characterizes the points of $g(s)$. So for the geodesics of our $F^n$ we obtain the following alternative

\begin{definition}\label{def:osc}
  In a Finsler space $F^n=(\RR^n,\ir F)$ with strictly convex geodesic spheres a geodesic consists of the osculation points of geodesic spheres with different radii centered at its two arbitrary points $a$ and $b$.
\end{definition}

\unit\label{ssec:Finsplane} In the above definition $a$ and $b$ were two arbitrary points of $g(s)$. So if we choose another point $a^*\in g(s)$ in place of $a\in g(s)$, then also $S_{a^*}(r^*)$, $r^*= \varrho^F(a^*,c)$ and $S_b(\bar r)$, $\bar r=\varrho^F(c,b)$ osculate each other at a point $c=g(r)$, and the three geodesic spheres $S_a(r)$, $S_b(r)$ and $S_{a^*}(r^*)$ have a common tangent plane $\varSigma_c$ at $c$. This leads to a further definition of the Finsler geodesics.

We show that at an arbitrary point $p_0$ of a Finsler geodesic $g(s)$ of our $F^n$ there exists a hyperplane $\varSigma_{p_0}$, such that the points of $g(s)$, and only these are centers of geodesic spheres osculating $\varSigma_{p_0}$ at $p_0$. Indeed, let $a$ be a point of $g(s)$, and $\varSigma_{p_0}$ the tangent plane of $S_a(r)$, $r=\varrho^F(a,p_0)$ at $p_0$: $T_{p_0}S_a(r)=\varSigma_{p_0}$. Let $p$ be an arbitrary point of $g(s)$. Then, as we have showed in the proof of Theorem~\ref{thm:convex}(a), $S_a(r)$, $r=\varrho^F(a,p_0)$ and $S_p(\bar r)$, $\bar r=\varrho^F(p,p_0)$ osculate at $p_0$, and $T_{p_0}S_a(r)=T_{p_0}S_p(\bar r)$. Thus any $p\in g(s)$ is a center of a geodesic sphere osculating $\varSigma_{p_0}$ at $p_0$. 

Conversely, we show that if a point $q\in\RR^n$ is the center of a geodesic sphere $S_q(r^*)$, $r^*=\varrho^F(q,p_0)$ osculating $\varSigma_{p_0}$ at $p_0$, then it belongs to $g(s)$. We supposed that $T_{p_0}S_q(r^*)=\varSigma_{p_0}$. Let $g^*(s)$ be the geodesic through $q$ and $p_0$. Its tangent at $p_0$ is $\dot g^*(0)$. We know that a geodesic sphere is perpendicular (in the sense of $F^n$) to any geodesic emanating from its center. Then $\dot g(0)\bot T_{p_0}S_a(r)=\varSigma_{p_0}$, and also $\dot g^*(0)\bot T_{p_0}S_q(r^*)=\varSigma_{p_0}$. Thus (with an appropriate orientation of $g^*(s)$) $\dot g(0)=\dot g^*(0)$. But from a point $p_0$ in a direction emanates a single geodesic. So $g(s)=g^*(s)$, and thus $q\in g(s)$. This yields the following equivalent

\begin{definition}     
In a Finsler space $F^n=(\RR^n,\ir F)$ with strictly convex geodesic spheres a geodesic consists exactly of the centers of those geodesic spheres, which osculate a hyperplane $\varSigma_{p_0}$ at its point $p_0$.
\end{definition}

\section{Osculation curves and quasigeodesics of differentiable distance spaces}\label{sec:osc}

\subsection{Differentiability conditions}
We investigate distance spaces $D^n=(\RR^n,\varrho)$ over the linear space $\RR^n$. Beyond (D1)--(D3) we suppose that the distance function $\varrho(x,y)$ is smooth

\begin{itemize}
\item[(D4)] \emph{$\varrho\colon\RR^n\times\RR^n\to\RR_+$, $(x,y)\mapsto\varrho(x,y)$ is $C^\infty$
     excluding $x=y$, where $\varrho$ is $C^0$.}
\end{itemize}

(D4) is an important property also of the Finsler distance function. In this paper distance spheres play an important role. We do not know of their earlier applications in such investigations. In order to avoid heaping of notations we use the same sign for geodesic- and distance-spheres. The context clearly shows, which one we speak of. A distance sphere centered at a point $a$ and having radius $r$ is defined and denoted by
\[
 S_a(r):=\{q\in\RR^n\mid \varrho(a,q)=r\}.
\]
This is a counterpart of geodesic- and Euclidean spheres. We also suppose that
\begin{itemize}
\item[(D5)] \begin{enumerate}
\item[(a)] \emph{distance spheres are strictly convex in $\RR^n$;}
\item[(b)] \emph{they have everywhere positive Gauss curvature;}
\item[(c)] \emph{for any point $p\in\RR$, the $1$-parameter family $S_p(r)$ ($r>0$) is diffeomorphic to that of Euclidean spheres in $\RR^n$ centered at $0$;}
\item[(d)] \emph{if a distance ball contains another one, their spheres $S_a(r)$ and $S_b(\bar r)$ can osculate at most in  one point $p$, and in any direction $v\in T_pS_a(r)=T_pS_b(\bar r)$, the normal curvature of $S_a(r)$ is smaller than that of $S_b(\bar r)$;}
 \end{enumerate}
\end{itemize}
From this it also follows that distance spheres osculate their tangent planes exactly in the first order. (D1)--(D4) and (D5d) are satisfied by any reversible $F^n$. 
We suppose that our $F^n=(\RR^n,\ir F)$ are reversible, and satisfy still (D5a).
[We remark that (D5a) often can be replaced by the property
\begin{itemize}
 \item [(D5a$^*$)] if two distance spheres have no common inner point, then their boundaries have at most one common (osculation) point. 
\end{itemize}
In this case the base manifold of $D^n$ need not be the $\RR^n$ but it can be any connected differentiable manifold $M^n$.]

\subsection{Construction of osculation curves}\label{ssec:osc}
We construct in a distance space through any arbitrary pair of points $a,b\in\RR^n$ a $1$-parameter set of osculation  points $o(r)$, $r\in(-\infty,\infty)$. We perform basically the same construction as in section \ref{ssec:Fosc}, but we replace geodesic spheres by distance spheres. Since $S_a(r)$ is a compact set, there are points $o_1(r)$ of $S_a(r)$, which are nearest to $b$, and points $o_2(r)\in S_a(r)$, which are furthest from $b$. $S_a(r)$ and $S_b(r_1)$, $r_1=\varrho(o_1(r),b)$ must be osculating and cannot be intersecting, for then would exist a $p\in S_a(r)$ which is nearer to $b$, than $o_1(r)$. Also $o_1(r)$ is unique because $S_a(r)$ and $S_b(r_1)$ are strictly convex.
 So $S_a(r)$ and $S_b(r_1)$ are osculating. So
 \begin{equation}\label{eq:o1}
 o_1(r)=(p\in S_a(r)\mid \varrho(p,b) \mbox{ is minimal}),\quad r\in\left[0,\infty\right),\quad\mbox{and }o_1(0)=a. 
 \end{equation}
For similar reasons also $S_a(r)$ and $S_b(r_2)$, $r_2=\varrho(o_2(r),b)$ must be osculating. Then $S_a(r)$ is in the inside of $S_b(r_2)$, and also $o_2(r)$ is unique by (D5d). 
 In this case
 \begin{equation}\label{eq:o2}
 o_2(r)=(p\in S_a(r)\mid \varrho(p,b) \mbox{ is maximal}),
 \end{equation}
Let us denote the parameter of $o_2$ by $-r(\leq0)$. Thus
\begin{equation}\label{eq:oo1o2}
 o(r):=\begin{cases}
        o_1(r),\quad r\geq 0\\
        o_2(r),\quad r\leq0
       \end{cases}\quad r\in(-\infty,\infty)
\end{equation}
is a $1$-parameter point set consisting exactly from the osculation points (from outside or inside) of distance spheres of different radii centered at $a$ and $b$, which are called the generator points of $o(r)$.
     
\subsection{Differentiability of $o(r)$}
\begin{proposition}\label{prop:osc}
  The point set $o(r)$, $r\in(-\infty,\infty)$ is a $C^\infty$ curve except at the generator points. 
\end{proposition}
 
\begin{proof}
Consider a polar coordinate system $(\varphi^\alpha,r)$, $\alpha=1,\dots,n-1$ on $\RR^n(x)$ centered at $a$, so that $r(p)=\varrho(a,p)$. Then $(\varphi^\alpha,r_0)$ gives a coordinate system $(\varphi^\alpha)$ on $S_a(r_0)$ for each $r_0>0$. The osculation point $o(r_0)$ with coordinates $(\varphi_0,r_0)$ is a minimum point of $\varrho_b:=\varrho(\cdot, b)$ on $S_a(r_0)$. Then
\begin{equation}\label{eq:rhobd}
 \left( \frac{\partial\varrho_b}{\partial\varphi^\alpha}(\varphi,r_0)\right)_{\varphi_0} = \frac{\partial\varrho_b}{\partial\varphi^\alpha}(o(r_0))=0.
\end{equation}
We show that
\begin{equation}\label{eq:rank}
\operatorname{rank}\frac{\partial^2\varrho_b}{\partial \varphi^\alpha\partial\varphi^\beta}(o(r_0))=n-1.
\end{equation}
Since $\varrho_b(\varphi,r)$ is $C^\infty$ also in $r$, from the implicit function theorem we obtain that $\varphi^\alpha(o(r))=\varphi^\alpha(r)$ are $C^\infty$ solutions of \eqref{eq:rhobd}, and thus $o(r)=(\varphi^\alpha(r),r)$ is a $C^\infty$ curve. Indeed, on the one hand we have
\begin{equation}\label{eq:ddrb}
 \frac{\partial^2 \varrho_b}{\partial \varphi^\alpha \partial \varphi^\beta}
 =\frac{\partial^2\varrho_b}{\partial x^i\partial x^j}\frac{\partial x^i}{\partial \varphi^\alpha}\frac{\partial x^j}{\partial\varphi^\beta}
 +\frac{\partial\varrho_b}{\partial x^i}\frac{\partial^2 x^i}{\partial \varphi^\alpha \partial\varphi^\beta},
\end{equation}
where $(x^i)$ are coordinates on $\RR^n$. On the other hand (denoting $ r$ by $\varphi^n$), $\frac{\partial\varrho_a}{\partial\varphi^\alpha}=\frac{\partial\varphi^n}{\partial\varphi^\alpha}=0$, and hence
\begin{equation}\label{eq:ddra}
 0=\frac{\partial^2\varrho_a}{\partial x^i\partial x^j}\frac{\partial x^i}{\partial\varphi^\alpha}\frac{\partial x^j}{\partial\varphi^\beta}
 +\frac{\partial\varrho_a}{\partial x^i}\frac{\partial^2 x^i}{\partial\varphi^\alpha\partial \varphi^\beta}.
\end{equation}
Since $o(r_0)$ is an osculation point we have
\begin{equation}\label{eq:SaSbtan}
\frac{\partial\varrho_a}{\partial x^i}(o(r_0))=\lambda\frac{\partial\varrho_b}{\partial x^i}(o(r_0)) 
\end{equation}
for some nonzero $\lambda\in\RR$. In fact, from \eqref{eq:SaSbtan} we can express $\lambda$ as
\[
 \lambda = \operatorname{sgn}(\lambda)\frac{\|\grad\varrho_a\|}{\|\grad\varrho_b\|}(o(r_0)),
\]
which is negative if $o(r_0)$ is between $a$ and $b$, and positive if it is outside.
So substituting \eqref{eq:ddra} and \eqref{eq:SaSbtan} into \eqref{eq:ddrb} yields
\[
 \frac{\partial ^2 \varrho_b}{\partial\varphi^\alpha\partial\varphi^\beta}(o(r_0))
 =\|\grad\varrho_b\|\left(\frac{1}{\|\grad\varrho_b\|}\frac{\partial^2\varrho_b}{\partial x^i\partial x^j}
 -\frac{\operatorname{sgn}(\lambda)}{\|\grad\varrho_a\|}  \frac{\partial^2\varrho_a}{\partial x^i\partial x^j}\right)
 \frac{\partial x^i}{\partial\varphi^\alpha}\frac{\partial x^j}{\partial\varphi^\beta}(o(r_0))
 \]
From this we get
\[
 v^\alpha v^\beta\frac{\partial ^2 \varrho_b}{\partial\varphi^\alpha\partial\varphi^\beta}(o(r_0))=\|\grad\varrho_b\|(\kappa_b(v) -\operatorname{sgn}(\lambda)\kappa_a(v)),\quad \forall v\in T_{o(r_0)}S_a(r_0),
\]
where $\kappa_b(v)$ and $\kappa_a(v)$ denote the normal curvatures of $S_b(r_1)$ and $S_a(r_0)$ resp. at $o(r_0)$ in the direction $v$. If $r_0\in(0,\varrho(a,b))$, then $\lambda$ must be negative, and we obtain immediately that $v^\alpha v^\beta\frac{\partial ^2 \varrho_b}{\partial\varphi^\alpha\partial\varphi^\beta}(o(r_0))>0$. If $r_0<0$ or $r_0>\varrho(a,b)$, then $\lambda$ is positive, but D5(d) implies that $\kappa_b(v) >\operatorname{sgn}(\lambda)\kappa_a(v)$, so again we have $v^\alpha v^\beta\frac{\partial ^2 \varrho_b}{\partial\varphi^\alpha\partial\varphi^\beta}(o(r_0))>0$.
\end{proof}

One can see from section~\ref{sec:pre} and Definition~\ref{def:osc} that in a Finsler space $F^n=(\RR^n,\ir F)$ with strictly convex geodesic spheres osculation curves are Finsler geodesics. So in these $F^n$ the differentiability of the osculation curves is immediate.

\subsection{Quasigeodesics} Two arbitrary points $\bar a$, $\bar b$ of an osculation curve $o(r;a,b)$ determine again an osculation curve $o(r;\bar a,\bar b)$. But $o(r;\bar a,\bar b)\neq o(r;a,b)$ in general. Namely $o(r;a,b)$ is defined by the osculating $S_a(r)$ and $S_b(r^*(r))$, while $o[r;\bar a,\bar b]$ is defined by $S_{\bar a}(r)$ and $S_{\bar b}(r^{**}(r))$, which are independent of $S_a(r)$ and $S_b(r^*(r))$. In the following we mostly will investigate distance spaces with the good property

\begin{itemize}
  \item[(D6)] \emph{Any pair of points of an osculation curve as generator points determines the same osculation curve:} 
  $\bar a,\bar b\in o(r;a,b)\Rightarrow o(r;a,b)=o(r;\bar a,\bar b)$.  
\end{itemize}
\begin{definition}\label{eq:qged}
   A \emph{quasigeodesic} $q(r;a,b)$ is an osculation curve $o(r;a,b)$ such that	 $o(r;a,b)=o(r;\bar a,\bar b)$ for any $\bar a,\bar b\in o(r;a,b)$.
\end{definition}
In other words: an osculation curve $o(r;a,b)$ is a quasigeodesic $q(r;a,b)$ if and only if its points are generators of the same osculation curve. This condition is satisfied in any $F^n=(\RR^n,\ir F)$ with strictly convex geodesic spheres. In a distance space satisfying \mbox{(D1)--(D6)}, every osculation curve is a quasigeodesic, so we obtain
\begin{theorem}\label{thm:qC}
  In a distance space $D^n=(\RR^n,\varrho)$ having properties \mbox{(D1)--(D6)} through any pair of points $a,b\in\RR^n$ there exists a unique $C^\infty$ quasigeodesic $q(r;a,b)$. 
\end{theorem}

The existence and uniqueness of $q(r;a,b)$ follows from section~\ref{ssec:osc} and (D6). Proposition~\ref{prop:osc} assures the differentiability of $q(r;a,b)$ except $a$ and $b$. However, we can choose any other two points $\bar a$ and $\bar b$ of $q(r;a,b)$, and they will generate the same quasigeodesic. Then Proposition~\ref{prop:osc} tells us that $q(r;a,b)$ is differentiable everywhere but $\bar a$ and $\bar b$. So $q(r;a,b)$ is differentiable also at $a$ and $b$.
Also it follows from section \ref{ssec:osc} that two quasigeodesics emanating from a point can not have other common points. The quasigeodesics emanating from a point cover $\RR^n$ one folded.

We show some properties of the quasigeodesics. Denote by $b$ the points of $S_a(r)$. $q(r;a,b)=q(r,b)$ are smooth quasigeodesics between $a$ and $b$ with $q(0,b)=a$, $q(1,b)=b$. Then $\frac{\partial}{\partial r}q(a,b)=\dot q(r,b)$ is the tangent vector field of the quasigeodesics $q(r,b)$. We know that $q(r_0,b)\in S_a(r_0)$, $\forall r_0\in\left(0,1\right]$. Hence there exists also $\frac{\partial }{\partial b}q(r,b)$. However we do not know whether the partial derivatives $\frac{\partial}{\partial r}q(r,b)$ and $\frac{\partial}{\partial b}q(r,b)$ are continuous both in $r$ and $b$. 
[We suppose $C_1$: the vector field $\dot q(r,b)$ is continuous on $[0,1]\times S_a(r)$, and $C_2$: from a point in a direction emanates at most one quasigeodesic. (This is satisfied by any Finsler space.)]

\begin{theorem}\label{thm:q}
  Under condition $C_1$ and $C_2$ in a distance space $D^n=(\RR^n,\varrho)$ with properties (D1)--(D6)
  \begin{enumerate}[label=(\alph*)]
   \item from any point $p_0$ in any direction $y$ emanates exactly one quasigeodesic
   \item the set of all quasigeodesics emanating from a point $p_0$ is diffeomorphic to the set of the rays out of the origin of a Euclidean space $E^n$
  \end{enumerate}
\end{theorem}

\begin{proof}
(a) Let us denote $\dot q(0,b)=\dot q_b$, and consider the map 
\[
 \mu^*\colon\{\dot q_b\mid b\in S_a(r)\}=: Q\longrightarrow S_a(r),\quad \dot q_b\mapsto b.
\]
Let us endow $T_a\RR^n$ with a Euclidean metric, and let $y_b$ be a Euclidean unit tangent vector of $q(r,b)$ at $a$. Then
\[
 \{y_b\mid b\in S_a(r)\}=:Y\subset S^{n-1}\subset E^n,
\]
where $S^{n-1}$ is the Euclidean unit sphere. The map
\[
 \mu\colon Y\to S_a(1),\quad y_b\mapsto b
\]
is univalent by $C_2$, and it is 1:1, for it is defined on the whole $Y$, and from $y_{b_1}\neq y_{b_2}$ follows $b_1\neq b_2$, namely otherwise the two quasigeodesics emanating from $a$ in the directions of $y_{b_1}$ and $y_{b_2}$ would intersect each other at $b_1=b_2$, what contradicts to Theorem~\ref{thm:qC}. Then there exists
\[
 \mu^{-1}\colon S_a(r)\longrightarrow Y,\quad b\mapsto y_b,
\]
which is also 1:1 and by $C_1$ it is continuous. Then also $\mu$ is continuous, for it is the inverse of the 1:1 and continuous map $\mu^{-1}$ on the compact $S_a(1)$. Then $Y$ and $S_a(1)$ are homeomorphic. We know that every $y_b$ is a Euclidean unit vector. Thus $Y\subset S^{n-1}$. If $Y$ is a proper part of $S^{n-1}$, then it has a non empty boundary, while $\mu Y=S_a(r)$ has not. In case of a homeomorphism this is not possible. Therefore $Y=S^{n-1}$. This yields the statement of (a).

(b) Let $q$ be a quasigeodesic of $D^n=(\RR^n,\varrho)$ emanating from $p_0$ and having a tangent $\dot q$ at $p_0$. Since a quasigeodesic emanating from $p_0$ has a single common point with each distance sphere $S_{p_0}(r)$, $r\in\RR_+$, $r$ can be a parameter of $q$: $q=q(r)$. Then the points of $q=q(r)\subset\RR^n$ can be represented  by the triples $(r;p_0,\dot q)$. Let us equip $T_{p_0}\RR^n$ again with a Euclidean metric having a polar coordinate system $(r,\varphi)\colon T_{p_0}\RR^n\cong E^n(r,\varphi)$. Since from $p_0$ in every direction emanates exactly one quasigeodesic, $\dot q$ can be replaced by $\varphi$. Thus $q(r)=(r;p_0,\dot q)=(r;p_0,\varphi_0)$. Then the 1:1 and differentiable mapping
\[
 \RR^n\supset q(r)=(r;p_0,\varphi_0)\to (r,\varphi_0)\subset E^n
 \]
is a diffeomorphism between $\RR^n$ and $E^n$, and it takes every quasigeodesic $q(r)=q(r;p_0,\varphi_0)$ emanating from $p_0$ into a ray of $E^n$, and conversely.
\end{proof}
\subsection{Another defining property of the quasigeodesics; perpendicularity}

We give still another equivalent definition of the quasigeodesics. In section \ref{ssec:Finsplane} we have proved that a Finsler geodesic consists of the centers of those geodesic spheres, which tangent a plane $\varSigma_{p_0}$ at its point $p_0$. Replacing geodesic spheres by distance spheres of our $D^n$, and Finsler geodesics by quasigeodesics of $D^n$, we can proceed just in the same way as we did it in section~\ref{ssec:Finsplane}. This rectifies the statement that the points of a quasigeodesic $q(r)$, $q(0)=p_0$ are centers of distance spheres osculating a hyperplane $\varSigma_{p_0}$ through $p_0$.

We claim that points not belonging to $q(r)$ do not have this property, and hence the above statement characterizes the quasigeodesics of a $D^n=(\RR^n,\varrho)$. In order to prove this, consider a point $c$ outside of $q(r)$: $c\notin q(r)$, and suppose that yet $S_c(r_1)$, $r_1=\varrho(c,p_0)$ tangents $\varSigma_{p_0}$ at $p_0$: $T_{p_0}S_c(r_1)=\varSigma_{p_0}$. Let $\hat q(r)$, $\hat q(0)=p_0$ be the quasigeodesic through $c$ and $p_0$. The distance spheres $S_{p_0}(r)$, $r>0$ intersect the two quasigeodesics at $q(r)$ and $\hat q(r)$ resp. Their tangent planes at these points are $T_{q(r)}S_{p_0}(r)$ and $T_{\hat q(r)}S_{p_0}(r)$ resp. with the property
\[
 \lim_{r\to0} T_{q(r)}=\lim_{t\to0} T_{\hat q(r)}S_{p_0}(r)=\varSigma_{p_0}.
\]
Let us perform the diffeomorphism $\RR^n\to E^n$ applied in the proof of Theorem~\ref{thm:q}(b). Then the distance spheres $S_{p_0}(r)$ will be taken into Euclidean spheres $S^{n-1}(r)\subset E^n$, and the quasigeodesics $q(r)$ and $\hat q(r)$ go over into rays $\lambda(r)$ and $\hat\lambda(r)$ resp.\ of $E^n$. The images of $T_{q(r)}S_{p_0}(r)$ and $T_{\hat q(r)}S_{p_0}(r)$  will be hyperplanes $\psi_{\lambda(r)}$ and $\hat\psi_{\hat\lambda(r)}$ perpendicular to $\lambda(r)$ and $\hat\lambda(r)$ resp. They have the same position $\psi$ and $\hat\psi$ along $\lambda(r)$ and $\hat\lambda(r)$ resp.  Then, because of the diffeomorphism $\RR^n\to E^n$ we would obtain
\[
 \lim_{r\to 0}\psi_{\lambda(r)}=\psi=\lim_{r\to 0}\hat\psi_{\hat\lambda(r)}=\hat\psi.
\]
But $\psi\neq\hat\psi$, for they are perpendicular to two different rays $\lambda(r)$ and $\hat\lambda(r)$ resp. So our supposition cannot be true, that is the centers of the distance spheres osculating $\varSigma_{p_0}$ at $p_0$ lie on the quasigeodesic $q(r)$. Thus we obtain for the quasigeodesics another equivalent 

\begin{definition}\label{df:qsigma}
  A quasigeodesic $q(r)$ of a $D^n=(\RR^n,\varrho)$ with properties $(D1)--(D6)$ consists exactly of the centers of those distance spheres, which tangent a hyperplane $\varSigma_{p_0}$ at its point $p_0$.
\end{definition}
Definition~\ref{df:qsigma} also yields that to the tangent $\dot q(0)$ of a quasigeodesic $q(r)$ with $q(0)=p_0$ there exists a hyperplane $\varSigma_{p_0}$. We call $\varSigma_{p_0}$ perpendicular (in the sense of $D^n=(\RR^n,\varrho)$) to $\dot q(0)$. This perpendicularity holds in Finsler spaces $F^n=(\RR^n,\ir F)$, but it lacks more good properties of the perpendicularity in Euclidean space.

In a $D^2=(\RR^2,\varrho)$ two not intersecting quasigeodesics can be said \emph{parallel} (in the sense of $D^2$). We show that to any quasigeodesic $q(t)$ through any point $p\notin q(t)$ there exists at least one quasigeodesic parallel to $q(t)$.

Let us apply the diffeomorphism used in the proof of Theorem~\ref{thm:q}(b). Then the half quasigeodesics emanating from $p$ and passing through $q(t)$ are rays $r(\tau,t)$ with parameter $\tau$, such that $r(0,t)=p$. Let us denote
\[
 \lim_{t\to-\infty}r(\tau,t)=r_-(\tau)\quad\mbox{and}\quad \lim_{t\to+\infty} r(\tau,t)=r_+(\tau).
\]
If the Euclidean angle
\[
 \alpha=\sphericalangle(\dot r_+(0),\dot r_-(0))\geq\pi,
\]
then there exists an $r(\tau,t_0)$, such that also its extension in the opposite direction $r^*(\tau,t_0)$ meets $q(t)$, and thus the full quasigeodesic $r(\tau,t_0)\cup r^*(\tau,t_0)$, which is a straight line, meets $q(t)$ twice, which contradicts to Theorem~\ref{thm:qC}. So we obtain that $\alpha\leq\pi$. But in this case there exists a straight line, i.e., a quasigeodesic through $p$, which does not meet $q(t)$.

\section{Relation to Finsler spaces}\label{sec:Fins}
\setcounter{unit}0

\subsection{$F^n$ determined by $D^n$}
Condition (F3) implies the strict convexity of the indicatrices. In the following theorem we will call an $F^n=(\RR^n,\ir F)$ a \emph{weak Finsler space} if instead of (F3) the indicatrix balls are only convex.
We formulate a condition for a distance space $D^n=(\RR^n,\varrho)$ under which we can construct a weak Finsler space $F^n=(\RR^n,\ir F)$ from $D^n$. 

Consider a polar coordinate system $(\varphi^\alpha,r)$ on $\RR^n$, such that $r$ is the Euclidean distance from the origin. Since all the points of $\RR^n\times\RR^n\setminus\ir D$ (where $\ir D$ is the diagonal $\{(x,y)\in\RR^n\times \RR^n\mid x= y\}$) can be written uniquely in the form \mbox{$(x,x+ry)$}, $x\in\RR^n$, $y\in S^{n-1}$, $r>0$, we obtain a coordinate system on $\RR^n\times\RR^n\setminus\ir D$ by setting
\begin{equation}\label{eq:sphcoord}
 (x,x+ry)\longmapsto (x^i,\varphi^\alpha(y),r).
\end{equation}
So we may use $(x^i,\varphi^\alpha,r)$ as a coordinate system on $\RR^n\times\RR^n\setminus\ir D$. In this setting, $r$ becomes the Euclidean distance $(p,q)\mapsto \|p-q\|$.

Our further condition on $D^n=(\RR^n,\varrho)$ is that
\begin{itemize}
\item[(D7)] \emph{in all coordinate systems of the form \eqref{eq:sphcoord}, all the partial derivatives of $\frac\varrho r$ with respect to $x^i$, $\varphi^\alpha$ and $r$ are bounded on a neighbourhood of $\ir D$.}
\end{itemize}

\begin{theorem}
\label{thm:2}
 A distance space $D^n=(\RR^n,\varrho)$ with properties (D1)--(D7) determines a weak Finsler space $F^n=(\RR^n,\ir F)$ by
 \begin{equation}\label{eq:distF}
  \ir F(x,y)=\lim_{r\to0^+}\frac1r \varrho(x,x+ry). 
 \end{equation}
\end{theorem}

\begin{proof}
The positive homogeneity of $\ir F$ is clear, because it is a directional derivative. To prove that it is $C^\infty$, define a function $h$ by
\[
 h(x,y,r):=\frac1r\varrho (x,x+ry).
\]
where $x\in\RR^n$, $y\in S^{n-1}$ and $r>0$. Then $h$ is smooth, and if we consider $(\varphi^\alpha)$ as coordinates on $S^{n-1}$, (D7) implies that $h$ has bounded partial derivatives as $r$ approaches $0$.  Thus its partial derivatives have continuous extensions to $\RR^n\times S^{n-1}\times \{r\in\RR\mid r\geq 0\}$. This implies that $h$ has a smooth extension $\bar h$ to $\RR^n\times S^{n-1}\times\RR$ (see \cite{seeley1964extension}). Therefore the limit
\[
 \lim_{r\to 0} h(x,\varphi^\alpha(y),r)=\lim_{r\to 0^+}\frac 1r\varrho(x,x+ry)=:\ir F(x,y)
\]
exists and depends smoothly on $x\in \RR^n$ and $y\in S^{n-1}$. Thus $\ir F$ is smooth on $\RR^n\times S^{n-1}$. It is also positive homogeneous so it is smooth on $T\RR^n\setminus\{0\}\cong \RR^n\times\RR^n\setminus\ir D$.

Finally we show that the indicatrix balls of $\ir F$ are convex. For each $x_0$ and $r>0$, there exists on $\RR^n$ a unique nonnegative and positive homogeneous function $f_r(x_0;y)$ such that its indicatrix $\{y\in\RR^n\mid f_r(x_0,y)=r\}$ is $S_{x_0}(r)$ From the homogeneity of $f_r$ we have $\lim_{r\to 0^+}\frac1r f_r(x_0;y)=\ir F(x_0,y)$. $\frac1rf_r(x_0,y)$ is convex, hence so is $\ir F(x_0,y)$. This is true for any $x_0\in\RR^n$.
\end{proof}

\begin{remark}\label{rmrk:rhoparc} 
All partial derivatives of $\ir F$ can be obtained as limits of partial derivatives of $\varrho$. First, notice that we may define $\ir F$ also by
 \begin{equation}\label{eq:Fdrho}
   \ir F(x,y)=\lim_{t\to0^+}\frac{\partial\varrho}{\partial r}(x,x+ty).
 \end{equation}
 Indeed, using Taylor's formula one can show that if $f$ is a smooth real function with $f(0)=0$, and $g(t)=f(t)/t$, then
 \begin{equation}\label{eq:pert}
  \lim_{t\to0^+}g^{(k)}(t)=\lim_{t\to0^+}\frac1{k+1}f^{(k+1)}(t).
 \end{equation}
 Setting $f(t):=\varrho(x,x+ty)$ the relation \eqref{eq:Fdrho} follows. 
 
 Similar relations can be obtained for the further derivatives of $\ir F$. First, we have
 \[
  \frac 1r\frac{\partial^l\varrho}{\partial\varphi^{\alpha_1}\dots \partial\varphi^{\alpha_l}}=\frac{\partial^l\varrho/r}{\partial\varphi^{\alpha_1}\dots \partial\varphi^{\alpha_l}},\quad l\in\{0,1,2,\dots\}.
 \]
 (D7) provides that the right-hand side is bounded, hence $\frac{\partial^l\varrho}{\partial\varphi^{\alpha_1}\dots \partial^{\alpha_l}}$ tends to zero along any series $(x_n,y_n)\to (x,x)\in\ir D$. Using \eqref{eq:pert} again, we get
 \begin{align*}
  \lim_{t\to0^+}\frac{\partial^{k+l}\varrho}{\partial\varphi^{\alpha_1}\dots \partial\varphi^{\alpha_l}\partial r^k}(x,x+ty)&=k\lim_{t\to0^+}\frac{\partial^{k+l-1}\varrho/r}{\partial\varphi^{\alpha_1}\dots \partial\varphi^{\alpha_l}\partial r^{k-1}}(x,x+ty)\\
  &=k\frac{\partial^{k+l-1}\ir F}{\partial\varphi^{\alpha_1}\dots \partial\varphi^{\alpha_l}\partial r^{k-1}}(x,y)
 \end{align*}
 for all $(x,y)\in\RR^n\times S^{n-1}$. Due to the smoothness of $\bar h$ used in the proof of Theorem~\ref{thm:2}, these limits remain true along any series $(x_n,y_n,t_n)\in\RR^n\times S^{n-1}\times\RR_+^*$ tending to $(x,y,0)$:
 \[
  \lim_{n\to\infty}\frac{\partial^{k+l}\varrho}{\partial\varphi^{\alpha_1}\dots \partial\varphi^{\alpha_l}\partial r^k}(x_n,x_n+t_ny_n)=k\frac{\partial^{k+l-1}\ir F}{\partial\varphi^{\alpha_1}\dots \partial\varphi^{\alpha_l}\partial r^{k-1}}(x,y).
 \]
\end{remark}

\begin{remark}
The Finsler space $F^n=(\RR^n,\ir F)$ given by \eqref{eq:distF} was derived from a $D^n=(\RR^n,\varrho)$. By \eqref{eq:Fdist}, this $F^n$ determines again a distance function $\varrho^F$, whose value along the Finsler geodesic $g(t)$ is the Finsler arc length:
\begin{equation}
 s_F(g(t))=\varrho^F(g(t_0),g(t)):=\int_{t_0}^t \ir F(g(\tau),\dot g(\tau)) d\tau.
\end{equation}
From this we obtain
\[
 \left.\frac d{dt}\right|_{t_0}\varrho^F(g(t_0),g(t))=\ir F(g(t_0),\dot g(t_0)).
\]
Thus
\[
 \varrho\Rightarrow \ir F\Rightarrow \varrho^F\Rightarrow\ir F\Rightarrow\dots
\]
However the distance function $\varrho$ of the $D^n=(\RR^n,\varrho)$, from which $\ir F$ was derived by \eqref{eq:distF} differs in general from $\varrho^F$ obtained from $\ir F$ by \eqref{eq:Fdist}, i.e., $\varrho\neq\varrho^F$ in general (see \cite{tamassy2008relation}). The reason for this is that in the construction of $\ir F(x,y)$ we used $\varrho(x_0,x)$ in an arbitrary small neighbourhood of $x_0$, in a germ $G_{x_0}$ (see \cite{tamassy2008relation}). So for two $D^n=(\RR^n,\varrho)$ and $\bar D^n=(\RR^n,\bar\varrho)$, $\ir F(x_0,y)=\bar{\ir F}(x_0,y)$ holds if $\varrho(x_0,x)$ and $\bar\varrho(x_0,x)$ equal in $G_{x_0}$, but $\varrho(x_0,x)$ and $\bar\varrho(x_0,x)$ may differ outside $G_{x_0}$. So in this case $\ir F(x_0,y)=\bar{\ir F}(x_0,y)$ $\forall x_0\in\RR^n$ but $D^n\neq \bar D^n$ (see also \cite{tamassy2013some}).
\end{remark}

\subsection{Arc length in $D^n$}
In a distance space one can define arc length $s_D(x(t))$ in the same way as in a Finsler space:
\begin{equation}\label{eq:rhodist}
 s_D(x(t)):=\lim_{\max d_it\to0}\sum\varrho(x(t_i),x(t_{i+1})),\quad d_it=t_{i+1}-t_i.
\end{equation}
(see also \eqref{eq:arcintdx}). Then we have the following

\begin{theorem}\label{thm:3}
 In a distance space $D^n=(\RR^n,\varrho)$ with properties (D1)--(D7) the arc length $s_D(x(t))$ of a curve $x(t)$ equals the Finsler arc length $s_F(x(t))$ of $x(t)$ in the weak Finsler space $F^n=(\RR^n,\ir F)$, where $\ir F$ is given by \eqref{eq:distF}.
\end{theorem} 

\begin{proof}
 For simplicity we assume that $t\in [0,1]$. We estimate the arc length with the subdivision $t_i=\frac iN$, $i\in\{0,\dots,N\}$ of $[0,1]$. Let $\varrho_i(t):=\varrho(x(t_i),x(t))$. Using Taylor's formula:
 \begin{align*}
  \varrho(x(t_i),x(t_{i+1}))&=\varrho_i(t_{i+1})=\varrho_i(t_i)+\frac1N\varrho_i'(t_i)+\frac12\frac 1{N^2}\varrho_i''(t_i^*),\quad t_i^*\in[t_i,t_{i+1}].
 \end{align*}
Then, since $\varrho_i(t_i)=0$,
 \begin{align*}
  s_D(x)&=\lim_{N\to\infty}\sum_{i=1}^N\varrho(x(t_i),x(t_{i+1}))=\lim_{N\to\infty}\sum_{i=1}^N\Big(\frac 1N\varrho_i'(t_i)+\frac12\frac 1{N^2}\varrho_i''(t_i^*)\Big).
 \end{align*}
  We show that $\varrho_i'(t_i)=\ir F(x(t_i),x'(t_i))$ and that $\varrho_i''(t_i^*)$ have a uniform bound as $t_i^*\to t_i$. Consider a coordinate system given by \eqref{eq:sphcoord} and set $x^\alpha=\varphi^\alpha(x)$, $x^n=r(x)$. Then
 \begin{align}\label{eq:rhoD1}
   &(\varrho_{x(t_i)}\circ x)'=\frac{\partial\varrho}{\partial\varphi^\alpha}(x(t_i),x)\frac{\partial x^\alpha}{\partial t}+\frac{\partial\varrho}{\partial r}(x(t_i),x)\frac{\partial x^n}{\partial t},\\\label{eq:rhoD2}
 &\begin{aligned}[t]
  (\varrho_{x(t_i)}\circ x)''&=\frac{\partial^2\varrho}{\partial\varphi^\beta\partial \varphi^\alpha}(x(t_i),x)\frac{\partial x^\alpha}{\partial t}\frac{\partial x^\beta}{\partial t}
  +\frac{\partial^2\varrho}{\partial r\partial \varphi^\alpha}(x(t_i),x)\frac{\partial x^\alpha}{\partial t}\frac{\partial x^n}{\partial t}\\
  &+\frac{\partial^2\varrho}{\partial \varphi^\alpha \partial r}(x(t_i),x)\frac{\partial x^\alpha}{\partial t}\frac{\partial x^n}{\partial t}
  +\frac{\partial^2\varrho}{\partial r\partial r}(x(t_i),x)\frac{\partial x^n}{\partial t}\frac{\partial x^n}{\partial t}\\
  &+\frac{\partial\varrho}{\partial \varphi^\alpha}(x(t_i),x)\frac{\partial^2 x^\alpha}{\partial t^2}
  +\frac{\partial\varrho}{\partial r}(x(t_i),x)\frac{\partial ^2x^n}{\partial t^2}.
 \end{aligned}
 \end{align}
 According to Remark~\ref{rmrk:rhoparc},
 \[
 \frac{\partial\varrho}{\partial\varphi^\alpha}(x(t_i),x(t))\to 0\quad \mbox{and} \quad \frac{\partial\varrho}{\partial r}(x(t_i),x(t))\to \ir F\left(x(t_i),\frac{1}{\|\dot x(t_i)\|}\dot x(t_i)\right)
 \]
 as $t\to t_i$. Substituting these into \eqref{eq:rhoD1} we obtain $\varrho_i'(t_i)=\ir F(x(t_i),x'(t_i))$. Furthermore, also from Remark~\ref{rmrk:rhoparc}, we know that
 \[
 \frac{\partial^2\varrho}{\partial\varphi^\beta\partial \varphi^\alpha}(x(t_i),x(t)),\quad \frac{\partial^2\varrho}{\partial \varphi^\alpha \partial r}(x(t_i),x(t)),\quad \frac{\partial^2\varrho}{\partial r\partial r}(x(t_i),x(t))
 \]
 are all bounded as $t\to t_i$. So we only need to check that $\frac{\partial x^\alpha}{\partial t}$, $\frac{\partial x^n}{\partial t}$, $\frac{\partial^2 x^\alpha}{\partial t^2}$ and $\frac{\partial ^2x^n}{\partial t^2}$ are also bounded. For simplicity we assume $t_i=0$, $x(0)=0$. Since $x'(0)\neq 0$, on a small open interval containing $0$, we can reparametrize $x(t)$ so that $r(x(0),x(t))=t$, $t>0$. From this it is clear that $\frac{\partial x^n}{\partial t}$ and $\frac{\partial ^2x^n}{\partial t^2}$ are bounded as $t\to 0^+$.

 To prove that $\frac{\partial x^\alpha}{\partial t}$ and $\frac{\partial^2 x^\alpha}{\partial t^2}$ are also bounded, set $h=\varphi^\alpha(x(0),\cdot)$ and consider the `projection' $u(t):=\frac 1tx(t)$ of $x(t)$ to the sphere $S^{n-1}$. The function $h$ is positive homogeneous of degree $0$, so $h\circ x=h\circ u$. Using \eqref{eq:pert} again we obtain that $u$, $u'$ and $u''$ can be extended to $0$. Since $h$ is smooth on $\ir U\subset S^{n-1}$,  $(h\circ u)'(t)$ and $(h\circ u)''(t)$ have to be bounded as $t\to 0^+$.
 
 So we can conclude the proof:
 \begin{align*}
 s_D(x)&=\lim_{N\to\infty}\sum_{i=1}^N\Big(\frac1N\ir F(x(t_i),x'(t_i))+\frac12\frac 1{N^2}\varrho_i''(t_i^*)\Big)\\
  &=\int \ir F(x(t),x'(t))dt+\lim_{N\to\infty}\frac 1N\sum_{i=1}^N\frac 1N\varrho_i''(t_i^*)=\int \ir F(x(t),x'(t))dt.
  \end{align*}
\end{proof}

\begin{theorem}\label{thm:4}
 (a) The arc length $s_D(x(t))$ of a curve $x(t)$ connecting the points $a$ and $b$ is in general greater than $\varrho(a,b)$.

 (b) Let $D^n=(\RR^n,\varrho)$ be a distance space satisfying (D1)--(D7) and $F^n=(\RR^n,\ir F)$ the weak Finsler space determined by \eqref{eq:distF} having a geodesic between any pair of points $a$, $b$. Then $\varrho=\varrho^F$, i.e., $D^n=F^n$ if and only if  $s_D(q[a,b])=\varrho(a,b)$ for every quasigeodesic $q[a,b]$ of $D^n$.	
\end{theorem} 

\begin{proof}
 (a) In the definition \eqref{eq:rhodist} of the arc length $s_D$ the distance function $\varrho$ is used only for nearby points $x_{i-1}$, $x_i$. 
 So $s_D$ is independent of the value of $\varrho$ for two distant points $a$, $b$.
 
 Furthermore, by condition (D3) we have $\varrho(x_0,x_1)+\varrho(x_1,x_2)\geq\varrho(x_0,x_2)$.
Also $\varrho(x_0,x_2)+\varrho(x_2,x_3)\geq\varrho(x_0,x_3)$, etc. and the sign $>$ can effectively occur. From these follows that $s_D(x(t))\geq\varrho(a,b)$.

 (b) Suppose that for all quasigeodesic $q[a,b]$ we have $s_D(q[a,b]))=\varrho(a,b)$. By part (a) and Theorem~\ref{thm:3}, for any other curve $x[a,b]$ connecting $a$ and $b$ we have $s_F(x[a,b])=s_D(x[a,b])\geq\varrho(a,b)$. This means that $q[a,b]$ has the smallest arc length also in $F^n=(\RR^n,\ir F)$. Thus $q[a,b]$ is a geodesic arc $g[a,b]$ of $F^n$, and  $\varrho^F(a,b)=s_F(g[a,b])=s_D(q[a,b])=\varrho(a,b)$.

Conversely, suppose that $\varrho=\varrho^F$. Then the distance spheres of $D^n$ coincide with the geodesic spheres of the Finsler space $F^n$ induced by the $D^n$. Let $g(t)$ be a geodesic curve of the induced Finsler space through $a$ and $b$, and $g(t_0)\in g[a,b]\subset g(t)$. Then 
\begin{equation}\label{eq:sDg}
   s_D(g(t))\overset{\textrm{Th.\ref{thm:3}}}=s_F(g(t))=\varrho^F(a,b)=\varrho(a,b),
\end{equation}
and the geodesic spheres $S_a(r)$ and $S_b(\bar r)$ through $g(t_0)$ [where $r=\varrho^F(a,g(t_0)$ and $\bar r=\varrho^F(b,g(t_0))$] osculate each other at $g(t_0)$. But $S_a(r)$ and $S_b(\bar r)$ are at the same time distance spheres of $D^n$ too. So $g(t_0)$ is a point of the osculation curve $o(t_i;a,b)$. This is true for any $g(t_0)\in g(t)$. Thus $g(t)=o(t;a,b)$. Also this is true in case of any $\bar a,\bar b\in g(t)$: $g(t)\in o(t;\bar a,\bar b)$. That is the points of $g(t)$ are the generators of the same osculation curve. Then $g(t)$ is a quasigeodesic $q(t)$ (see Definition~\ref{eq:qged}, and the paragraph after it). Then \eqref{eq:sDg} yields that $s_D(q[a,b])=\varrho(a,b)$.
\end{proof}

\subsection{Projectively flat $F^n=(\RR^n,\ir F)$}
The indicatrices of a reversible Finsler space lie in the tangent spaces, which are linear spaces equipped with a Minkowski metric. These indicatrices are symmetric in the sense of the linear (affine) tangent space. A geodesic sphere of $F^n$ on the base manifold is geodesically symmetric. This means that a metrical reflection through its center takes the geodesic sphere into itself. However in our $F^n=(\RR^n,\ir F)$ we can speak of the symmetry of a geodesic sphere also in the sense of the affine space $\RR^n$, or -- if we endow $\RR^n$ with a Euclidean metric -- we can speak of symmetry in the Euclidean sense. We will use the symmetry of a geodesic sphere in this last sense.

In an $F^n=(\RR^n,\ir F)$ from the symmetry of the indicatrices (from the reversibility of the metric) does not follow the symmetry of the geodesic spheres. The symmetry of the geodesic spheres is a more restrictive condition.

\begin{theorem}\label{thm:5}
  A projectively flat reversible Finsler space $F^n=(\RR^n,\ir F)$ with symmetric geodesic spheres is a Minkowski space.
\end{theorem}

\begin{proof}
If $F^n=(\RR^n,\ir F)$ is projectively flat, a geodesic $g(t)$ parametrized by arc-length is of the form $g(t)=g(0)+f(t)\dot g(0)$, where $f(t)$ is a continuous strictly increasing function. We say that $F^n$ is projectively flat in a parameter preserving manner, if all such geodesics have constant Euclidean speed, that is, $f(t)$ is an affine function. In this case $F^n$ is a Minkowski space (see \cite{binh2013projectively}  Th.~5).
So we have to show only that from the symmetry of the geodesic spheres follows that $f(t)$ is affine for any geodesic.

Let $g(t)$ be a geodesic of a projectively flat $F^n=(\RR^n,\ir F)$. Then $g(t)=g(0)+f(t)\dot g(0)$. We can suppose that the Euclidean norm of $\dot g(0)$ is $1$: $\|\dot g(0)\|=1$.

The Euclidean distance between two points $g(t_2)$ and $g(t_1)$, $t_2>t_1$, of a geodesic $g(t)$ is
\begin{equation}\label{eq:flat}
 \varrho^E(g(t_2),g(t_1))=\|(g(0)+f(t_2)\dot g(0))-(g(0)+f(t_1)\dot g(0))\|=(f(t_2)-f(t_1)).
\end{equation}
Geodesics minimize distance locally, so if the interval $I$ is small enough, we have $\varrho^F(g(t_2),g(t_1))=|t_2-t_1|$ for all $t_1,t_2\in I$. Let $\alpha,\beta\in I$, $\beta>\alpha$, and set $a=g(\alpha)$, $b=g(\beta)$ and $m:=g(\frac{\beta-\alpha}2)$. Since $F$ is reversible, $a$ and $b$ both are on the distance sphere $S_m(\frac{\beta-\alpha}2)$. However,  $S_m(\frac{\beta-\alpha}2)$ is symmetric in $\RR^n$ by our assumption, so we have $\varrho^E(a,m)=\varrho^E(b,m)$. Then \eqref{eq:flat} gives
\[
 f\left(\frac{\beta-\alpha}2\right)-f(\alpha)=f(\beta)-f\left(\frac{\beta-\alpha}2\right),
\]
and hence
\[
 \frac{f(\beta)+f(\alpha)}2=f\left(\frac{\beta-\alpha}2\right).
\]
Since $f$ is continuous and $\alpha,\beta\in I$ are arbitrary, $f$ must be affine on $I$.
\end{proof}

\begin{corollary}
 A Riemannian space over $\RR^n$ with constant curvature having in $\RR^n$ symmetric geodesic spheres is Euclidean.
\end{corollary}

\begin{proof}
 A Riemannian space of constant curvature is projectively flat. If it has symmetric geodesic spheres, then it is Minkowskian by Theorem~\ref{thm:5}. But a Minkowskian Riemannian space is Euclidean.
\end{proof}


\bibliographystyle{siam2}      
\bibliography{finsler2}   

\end{document}